\documentclass[11pt]{amsart}
\usepackage{hyperref}
\newtheorem{lemma}{Lemma}

\begin{document}

\title{The growth of digital sums of powers of two}
\author{David G Radcliffe}
\date{\today}
\maketitle
\thispagestyle{empty}

In this note, we give an elementary proof that $s(2^n) > \log_4 n$ for all $n$, 
where $s(n)$ denotes the sum of the digits of $n$ written in base 10.
In particular, $\lim_{n\to\infty} s(2^n) = \infty$.

The reader will notice that this lower bound is very weak. 
The number of digits of $2^n$ is $\lfloor n \log_{10} 2\rfloor + 1$, 
so it is natural to conjecture that
$$
\lim_{n\to\infty} \frac{s(2^n)}{n} = \frac92 \log_{10} 2.
$$
However, this conjecture remains open\cite{oeisA001370}.

In 1970, H.\,G.\,Senge and E.\,G.\,Straus proved that 
the number of integers
whose sum of digits is less than a fixed bound with respect to the bases $a$ and $b$
is finite if and only if $\log_b a$ is rational\cite{MR0340185}.
As the sum of the digits
of $a^n$ in base $a$ is 1, this result implies that
$$\lim_{n\to\infty} s(a^n) = \infty$$
for all positive integers $a$ except powers of 10.
This work was extended by C.\,L.\,Stewart, who gave an effectively
computable lower bound for $s(a^n)$\ \cite{MR586115}. However,
this lower bound is asymptotically weaker than our bound, and Stewart's proof relies on deep
results in transcendental number theory.

We begin with two simple lemmas.

\begin{lemma}
Every positive integer $N$ can be expressed in the form
\begin{equation*}
     N = \sum_{i=1}^m d_i\cdot 10^{e_i}
\end{equation*}
where $d_i$ and $e_i$ are integers so that $1\le d_i\le 9$ and
$$0 \le e_1 < e_2 < \cdots < e_m.$$
Furthermore,
$$s(N) = \sum_{i=1}^m d_i \ge m. $$
\end{lemma}

\begin{proof}
The proof is by strong induction on $N$. The case $N < 10$ is trivial.
Suppose that $N \ge 10$. By the division algorithm, there exist integers $n\ge1$ and $0\le r\le9$ so that $N = 10n + r$. By the induction hypothesis, we can express $n$ in the form
$$ n = \sum_{i=1}^m d_i\cdot 10^{e_i}.$$
If $r = 0$, then 
$$
N = \sum_{i=1}^m d_i \cdot 10^{e_i+1}
$$
and if $r > 0$ then
$$
N = r\cdot 10^0 + \sum_{i=1}^m d_i \cdot 10^{e_i+1}.
$$
 In either case, $N$ has an expression of the required form.
\end{proof}

\begin{lemma}
Let $2^n = A + B\cdot 10^k$ where $A, B, k, n$ are positive integers and $A < 10^k$.
Then $A \ge 2^k$.
\end{lemma}

\begin{proof}
Since $2^n > 10^k > 2^k$, it follows that $n > k$, so $2^k$ divides $2^n$. But $2^k$ also divides $10^k$, therefore $2^k$ divides $A$. But $A > 0$, so $A \ge 2^k$.
\end{proof}

We use these lemmas to establish a lower bound on $s(2^n)$. Write
$$2^n = \sum_{i=1}^m d_i \cdot 10^{e_i}$$
so that the conditions of Lemma 1 hold, and let $k$ be an integer between 2 and $m$.
Then $2^n = A + B \cdot 10^{e_k}$ where
$$
A = \sum_{i=1}^{k-1} d_i\cdot10^{e_i}
$$
and
$$
B = \sum_{i=k}^m d_i\cdot 10^{e_i - e_k} .
$$
Since $A < 10^{e_k}$, Lemma 2 implies that $A \ge 2^{e_k}$.
Therefore,
$$
2^{e_k} \le A < 10^{e_{k-1}+1}
$$
which implies that
$$
e_k \le \lfloor(\log_2 10) (e_{k-1} + 1)\rfloor.
$$

We prove that $e_k < 4^{k-1}$ for all $k$. It is clear that $e_1 = 0$, else $2^n$ would be divisible by 10. From the inequality above, we have $e_2 \le 3$, $e_3 \le 13$, $e_4 \le 46$, $e_5 \le 156$, and $e_6\le 521$.
If $k \ge 7$ then $e_{k-1} \ge 5$, so
\begin{align*}
e_k &< (\log_2 10) e_{k-1} + (\log_2 10) \\
         &< \frac{10}3 e_{k-1} + \frac{10}3 \\\
         &\le \frac{10}3 e_{k-1} + \frac23 e_{k-1}  \\
         &= 4e_{k-1}.
\end{align*}
Therefore, $e_k < 4^{k-1}$ for all $k$, by induction.

We are now able to prove the main result. Note that
$$
2^n < 10^{e_m+1} \le 10^{4^{m-1}}
$$
since $10^{e_m}$ is the leading power of 10 in the decimal expansion of $2^n$.

Taking logarithms gives
\begin{align*}
4^{m-1} &> n \log_{10} 2  \\
4^{m-1} &> n/4\\
4^m &> n \\
m &> \log_4 n
\end{align*}
hence
$$s(2^n) > \log_4 n .$$
In particular, $$\lim_{n\to\infty} s(2^n) = \infty.$$

\bibliographystyle{plain}

\bibliography{mybib}
\end{document}